\documentclass{article}
\usepackage[a4paper, portrait, margin=1in]{geometry}
\usepackage{authblk}
\usepackage[english]{babel}
\usepackage[utf8]{inputenc}
\usepackage[T1]{fontenc}
\usepackage{amsmath, amsthm, amssymb}
\usepackage{physics}
\usepackage{array}
\usepackage{helvet}
\usepackage{cite}
\usepackage{etoolbox}
\usepackage{graphicx}
\usepackage{titlesec}
\usepackage{booktabs}
\usepackage{xcolor} 
\usepackage{scrextend}
\usepackage{fancyhdr}
\usepackage{graphicx}
\usepackage[colorlinks, citecolor=red, urlcolor=blue]{hyperref}

\newtheorem{lemma}{\scshape Lemma}
\newtheorem{proposition}{\scshape Proposition}
\newtheorem{definition}{\scshape Definition}
\newtheorem{remark}{\scshape Remark}
\newtheorem{theorem}{\scshape Theorem}

\title{Exponential Stability of a Degenerate Euler-Bernoulli Beam with Axial Force and Delayed Boundary Control}

\author[1]{Ben Bakary Junior SIRIKI \footnote{Corresponding author: \url{benbjsiriki@gmail.com}}}
\author[2]{Adama COULIBALY}

\affil[1]{Université Nangui Abrogoua, Abidjan, Côte D'Ivoire}
\affil[2]{Université Félix Houphouët-Boigny, Abidjan, Côte D'Ivoire}

\date{}

\providecommand{\keywords}[1]
{
	\small
	\textbf{Keywords:} Degenerate Euler-Bernoulli beam, axial force, time delay, exponential stability, Lyapunov functional
}

\begin{document}
	
\maketitle
	
	\begin{abstract}
This work investigates the global exponential stabilization of a degenerate Euler-Bernoulli beam subjected to a non uniform axial force and a delayed feedback control.	
First, we address the well-posedness of the system by constructing an appropriate energy space in weighted Sobolev settings. Using L\"umer-Phillips theorem, we prove that the linear operator associated with the problem generates a $\mathcal{C}_0$-semigroup of contractions.
Next, we establish the uniform exponential stability of the system. By constructing a novel Lyapunov functional incorporating weighted integral terms, we demonstrate that the energy of the system exponentially decays to zero and derive a precise decay rate estimate.
This work provides a significant extension to the stability theory for complex distributed parameter systems.
	\end{abstract}
	

\section{Introduction}
	The growing need for the accurate modeling, prediction, and control of complex physical systems has made stabilization theory a critical research direction within the scientific community. In this context, the stability analysis of Euler-Bernoulli beam models has received considerable attention from researchers in recent decades. These models are central to various engineering applications: spanning civil engineering \cite{du2022, fryba1996}, aeronautics \cite{balakrishnan1986, vinod2007, chakravarthy2010, kundu2017}, robotics \cite{kundu2017, luo1993}, and nanotechnology \cite{turner2001}. However, to the best of our knowledge, studies addressing this question remain scarce in a context that simultaneously integrates the degeneracy of the flexural rigidity, the axial force, and control laws incorporating a time delay.	
	
	This article addresses the stability analysis of an Euler-Bernoulli beam model with a degenerate flexural rigidity under axial force effects. The beam is characterized by a dynamic boundary condition at one end involving an operator $\mathcal{B}$, whereas the other end is controlled via a feedback law with a constant time-delay. The governing equations describing this model are presented below:
	\begin{equation} \label{eq:p1}
	\left\{
	\begin{array}{>{\displaystyle}l}
	 u_{tt} + \left( \sigma(x) u_{xx} \right)_{xx} - \left( q(x) u_x \right)_x = 0 , \quad \mathcolor{blue}{0 < x < 1}, \; t > 0, \\
	u(0,t) = \mathcal{B} u(0,t) = 0, \quad t > 0,  \\
	(\sigma (x) u_{xx})(1,t) + u_{xt}(1,t) = 0, \quad t > 0,  \\
	\left( \sigma (x) u_{xx} \right)_x(1,t) - q(1) u_x(1,t) = \kappa_1 u_t(1,t) + \kappa_2  u_t(1,t-\tau), \quad t > 0, \\
	u(x,0) = u_0(x), \, u_t(x,0) = u_1(x), \quad 0 \leqslant x \leqslant 1, \\
	u_t(1, t-\tau) = f_0(t-\tau), \quad 0 < t < \tau,
	\end{array}
	\right.
\end{equation}
where $u(x,t)$ denotes the transverse displacement of the beam at position $x$ and time $t$. The constants $\kappa_1 \geqslant 0$ and $\kappa_2 \neq 0$ represent the control gains. In addition, $\tau > 0$ represents the time delay, and $u_0$, $u_1$, $f_0$ are the initial data. The axial force function $q$ satisfies the following conditions:
\begin{equation} \label{eq:1}
	\left\{
	\begin{array}{>{\displaystyle}l}
	 q \in W^{1, \infty}(0,1), \\
	 0 < q_0 \leqslant q(x) \leqslant q_1, \quad 0 \leqslant x \leqslant 1,\\
	 \left| q'(x) \right| \leqslant q_2, \quad 0 \leqslant x \leqslant 1,
	\end{array}
	\right.
\end{equation}
and the flexural rigidity function $\sigma: \, [0, 1] \to \mathbb{R}_+$  is such that:
\begin{equation} \label{eq:2}
	\left\{
	\begin{array}{>{\displaystyle}l}
	 \sigma(0) = 0 \, \text{ and } \sigma(x) > 0, \; 0 < x \leqslant 1 \\
	 \iota_\sigma := \underset{x \in (0, 1]}{\sup} \frac{x |\sigma'(x)|}{\sigma(x)} < 2.
	\end{array}
	\right.
\end{equation}
and satisfies one of the following conditions:
\begin{align}
\sigma \in C[0,1] \cap C^1(0,1], \; \iota_\sigma \in (0, 1) \label{eq:3}\\
\sigma \in C^1[0,1], \; \iota_\sigma \in [1, 2). \label{eq:4}
\end{align}
The function $\sigma$ is defined as weakly degenerate, (WD) for short (respectively strongly degenerate, (SD) for short) when it satisfies conditions \eqref{eq:2} and \eqref{eq:3} (respectively \eqref{eq:2} and \eqref{eq:4}). In addition, due to the degeneracy of the problem at the boundary $x=0$, we emphasize that the operator $\mathcal{B}$ is defined as follows:
\begin{gather} \label{eq:6}
\mathcal{B} u(x,t) := \left\{
\begin{array}{l}
u_x(x,t) \quad \text{ if } \sigma \text{ is (WD),} \\ ~ \\
\left(\sigma(x) u_{xx}\right)(x,t) \quad \text{ if } \sigma \text{ is (SD).}
\end{array}
\right.
\end{gather}
Following the definition of the operator $\mathcal{B}$, we provide the physical interpretation of the boundary conditions $\eqref{eq:p1}_2$. Specifically, two distinct regimes are considered:
\begin{itemize}
\item in the weak degeneracy (WD) case, the beam  is \textit{clamped} at the boundary $x=0$. This implies that the beam can neither move vertically nor rotate at this point;
\item in the strong degeneracy (SD) case, when $u_{xx}(0,t) = 0$, the beam is \textit{hinged} (or \textit{simply supported}) at its end $x=0$.  In this configuration, the beam is free to rotate but cannot deflect vertically or support any bending moment.
\end{itemize}
Numerous studies have investigated the stability of Euler-Bernoulli beams. However, the vast majority of existing models neglect the effects of axial force (see for example \cite{han2016, kugi2005, shang2012, wang2005, teya2023, marc2017, guo2008, baysal2024, luo2012}) compared to those that take it into account (see \cite{benamara2025, ledkim2023}).
Furthermore, we emphasize that all models employed in the aforementioned references are non-degenerate; consequently, the structural wear of the beam over time or material defects are not accounted for, which can significantly alter the long-term stability of aging infrastructure.
The literature addressing such degenerate cases is remarkably scarce; among the few notable exceptions are \cite{camasta2024, camasta2025, salhi2025}. However, none of these works simultaneously account for the destabilizing effect of a non-uniform axial force combined with a delayed feedback.
 
	A crucial question is whether a robust analytical framework can be established to ensure the well-posedness of this degenerate Euler-Bernoulli beam problem under a non-uniform axial force and a delayed boundary control. Since even small time delays are known to induce instabilities, as highlighted in the papers \cite{datko1991, datko1993, han2016}, is it still possible to achieve stability for system \eqref{eq:p1}? 
This paper fills a gap in the literature by providing a unified stabilization analysis for a degenerate beam model that incorporates both variable axial force and boundary control delay.

	Inspired by anterior works \cite{camasta2025a, camasta2024, camasta2025, liao2024, alabau2006, alabau2017} and taking into account the degeneracy of the flexural rigidity, the presence of the axial force and the time delay in the boundary control, we first design an appropriate functional framework to answer these questions. Next, under the crucial condition that the control gains satisfy the following inequality:
\begin{gather}
\kappa_1 > \left|\kappa_2\right|, \label{eq:0}
\end{gather}
we demonstrate that problem \eqref{eq:p1} is well-posed using a semigroup theory approach. We then proceed to establish the exponential stability of system \eqref{eq:p1} using the Lyapunov method. Our approach relies on two key steps: the construction of a suitable Lyapunov functional and the derivation of refined estimates for the functional and its time derivative.

The rest of this paper is organized as follows. Section 2 reviews several preliminary results useful for the subsequent analysis. Next, in Section 3, we reformulate problem \eqref{eq:p1} into an abstract evolution problem in a suitable Hilbert state space. Then, under certain conditions, we demonstrate that the problem is well-posed in the sense of semigroups. In Section 4, we show that system \eqref{eq:p1} is exponentially stable using the Lyapunov method under the same conditions. Finally, Section 5 summarizes the main results of this article and discusses future research perspectives.
 
\section{Weighted functional spaces}

In order to study system \eqref{eq:p1}, let us introduce some Hilbert spaces with the related inner products (see \cite{camasta2025a, camasta2024, camasta2025, alabau2006, alabau2017, salhi2025}. Let
\begin{equation} \label{eq:7}
	V_\sigma^2(0, 1) = \left\{
	\begin{array}{>{\displaystyle}l}
	 \Big\{ u \in H^1(0,1): \, u' \text{ absolutely continuous  in } [0,1], \\
	 \qquad \qquad \qquad \quad \sqrt{\sigma} u'' \in L^2(0,1) \Big\}, \; \text{ if } \underline{\sigma \text{ is (WD)}}; \\
	 ~ \\
	 \Big\{ u \in H^1(0,1): \, u' \text{ locally absolutely continuous  in } (0,1], \\
	 \qquad \qquad \qquad \quad \sqrt{\sigma} u'' \in L^2(0,1) \Big\}, \; \text{ if } \underline{\sigma \text{ is (SD)}}. 
	\end{array}
	\right.
\end{equation}
On $V_\sigma^2(0, 1)$, we consider the inner product defined as: 
\begin{align}
\prec u, v \succ_\sigma \, := \int_{0}^{1} u(x) v(x) \, dx + \int_{0}^{1} u'(x) v'(x) \, dx + \int_{0}^{1} \sigma(x) u''(x) v''(x) \, dx, \label{eq:8a}
\end{align}
for all $u, \, v \in V_\sigma^2(0, 1)$, which induces the norm:
\begin{align}
\lVert u \rVert^2_{2, \sigma} & := \lVert u \rVert_{L^2(0,1)}^2 + \lVert u' \rVert_{L^2(0,1)}^2 + \lVert \sqrt{\sigma} u'' \rVert_{L^2(0,1)}^2. \label{eq:8}
\end{align}
In addition, we have the following functional space:
\begin{align}
H_\sigma^2(0,1) := \left\{u \in V_\sigma^2(0, 1): \; u(0) = 0\right\}. \label{eq:9}
\end{align}
$H_\sigma^2(0,1)$ is a linear subspace of $V_\sigma^2(0,1)$.  Endowed with the inner product defined in \eqref{eq:8a}, $V_\sigma^2(0, 1)$ and $H_\sigma^2(0,1)$ are Hilbert spaces. 

\begin{remark} \label{r1}
Let $u \in H_\sigma^2(0,1)$. As $u(0) = 0$, we have for every $x \in (0, 1]:$
\begin{gather}
\lVert u \rVert_{L^2(0,1)}^2 \leqslant \lVert u' \rVert_{L^2(0,1)}^2 \leqslant \frac{1}{q_0} \left\lVert \sqrt{q} u' \right\rVert_{L^2(0,1)}^2 \quad \text{ and } \label{eq:10}\\
  \left|u'(1)\right|^2 \leqslant 2 \left( \frac{1}{q_0} \left\lVert \sqrt{q} u' \right\rVert_{L^2(0,1)}^2 + \frac{\left\lVert \sqrt{\sigma} u'' \right\rVert_{L^2(0,1)}^2}{\sigma(1)(2-\iota_\sigma)} \right). \label{eq:11}
\end{gather}
\end{remark}
Next, we introduce the following functional space:
\begin{align}
Q_\sigma(0,1) := \left\{u \in H_\sigma^2(0,1): \, \sigma u'' \in H^2(0,1) \right\}. \label{eq:1.12}
\end{align}
By taking into account the boundary conditions $\eqref{eq:p1}_2$, we conclude this functional setting by the following ones:
\begin{align}
	H_{\sigma,0}^2(0,1) = \left\{
	\begin{array}{>{\displaystyle}l}
	 \left\{ u \in H^2_\sigma(0,1): \, u'(0) = 0  \right\}, \; \text{ if } \underline{\sigma \text{ is (WD)}},\\
	 H^2_\sigma(0,1), \; \text{ if } \underline{\sigma \text{ is (SD)}};
	\end{array}
	\right. \label{eq:13}\\
	Q_{\sigma,0}(0, 1) = \left\{
	\begin{array}{>{\displaystyle}l}
	 \left\{ u \in Q_\sigma(0,1): \, u'(0) = 0  \right\},\; \text{ if } \underline{\sigma \text{ is (WD)}},\\
	 \left\{ u \in Q_\sigma(0,1): \, (\sigma u'')(0) = 0 \right\}, \; \text{ if } \underline{\sigma \text{ is (SD)}}.
	\end{array}
	\right. \label{eq:14}
\end{align}

\section{Global existence of solution}\label{sec3}

This section is devoted to the well-posedness of system \eqref{eq:p1}. We first reformulate it into an augmented system. Next, we transform the resulting problem into an abstract Cauchy problem that we solve via a semigroup approach \cite{pazy1983, engel2000}.

	\subsection{Augmented Model}
We reformulate system \eqref{eq:p1} into an augmented system. To do this, we use the following result.

\begin{proposition} \label{p1}
Let $w$ be a function defined by:
\begin{align}
w(s,t) := u_t(1, t- s \tau), \quad 0 < s < 1, \;\; t > 0. \label{eq:15}
\end{align}
Then $w$ satisfies the following system:
\begin{equation} \label{eq:16}
	\left\{
	\begin{array}{>{\displaystyle}l}
	\tau w_t + w_s = 0, \quad 0 < s < 1, \; t > 0, \\
	w(0, t) = u_t(1, t), \quad t>0, \\
	w(1, t) = u_t(1, t-\tau), \quad t>0, \\
	w(s, 0) = f_0(-s \tau), \quad 0 < s < 1.
	\end{array}
	\right.
\end{equation}
\end{proposition}
Using Proposition \ref{p1}, system \eqref{eq:p1} is equivalent to:
\begin{equation} \label{eq:pb1}
	\left\{
	\begin{array}{>{\displaystyle}l}
	u_{tt} + \left( \sigma(x) u_{xx} \right)_{xx} - \left( q(x) u_x \right)_x = 0 , \quad 0 < x < 1, \; t > 0, \\
	\tau w_t + w_s = 0, \quad 0 < s < 1, \; t > 0, \\
	u(0,t) = \mathcal{B} u(0,t) = 0, \quad t > 0,  \\
	(\sigma(x) u_{xx})(1,t) + u_{xt}(1,t) = 0, \quad t > 0,  \\
	\left( \sigma (x) u_{xx} \right)_x(1,t) - q(1) u_x(1,t) = \kappa_1 u_t(1, t) + \kappa_2  w(1,t), \quad t > 0, \\
	w(0, t) = u_t(1, t), \quad t>0, \\
	w(1, t) = u_t(1, t-\tau), \quad t>0,\\
	u(x,0) = u_0(x), \, u_t(x,0) = u_1(x), \quad 0 \leqslant x \leqslant 1, \\
	w(s, 0) = f_0(-s \tau), \quad 0 < s < 1.
	\end{array}
	\right.
\end{equation}

	\subsection{Well-posedness of the problem}

	In this subsection, we formulate the abstract Cauchy problem associated with system \eqref{eq:pb1} in a adequate Hilbert state space. Next, by using Hille-Yosida Theorem (see \cite{haim1983}), we prove that system \eqref{eq:p1} is well-posed.
	
	Let $\left(\mathcal{H}, \prec \cdot, \cdot \succ\right)$ be a Hilbert space defined by:
\begin{gather}
\mathcal{H} := H^2_{\sigma,0}(0,1) \times L^2(0,1) \times L^2(0,1)  \label{eq:17} \\
\begin{gathered}
\prec Y_1, Y_2 \succ := \int_0^1 v_1(x) \, v_2(x) \, dx + \int_0^1 q(x) \, u_{1,x}(x) \, u_{2,x}(x) \, dx \\
+ \int_0^1 \sigma(x) \, u_{1,xx}(x) \, u_{2,xx}(x) \, dx + \gamma \tau \int_0^1 w_1(s) \, w_2(s) \, ds.  \label{eq:18}
\end{gathered}
\end{gather}
where $Y_i = (u_i, \, v_i, \, w_i), \, i = 1, 2$, the parameters $\kappa_1, \, \kappa_2$ satisfy \eqref{eq:0} and a positive constant $\gamma$ such that:
\begin{gather}  \label{eq:19}
\left|\kappa_2\right| < \gamma < 2 \kappa_1 - \left|\kappa_2\right|.
\end{gather}

	Let $\mathbb{A}$ be an operator with domain $\mathcal{D}(\mathbb{A}) \subset \mathcal{H}$ defined as follows:
\begin{gather}
\begin{gathered}
\mathcal{D}(\mathbb{A}) := 
\left\{
\begin{array}{l}
(u, v, w) \in Q_{\sigma,0}(0,1) \times H^2_{\sigma,0}(0,1) \times H^1(0,1) \bigg| \, w(0) = v(1),\\
(\sigma u'')(1) + v'(1) = 0, \, (\sigma u'')'(1) - q(1)u'(1) = \kappa_1 v(1) + \kappa_2 w(1)
\end{array}
\right\}
\end{gathered} \label{eq:20}  \\
\mathbb{A} \begin{pmatrix}u \\ v \\ w\end{pmatrix} 
:= \begin{pmatrix}v \\ (q u')' - (\sigma u'')'' \\ - \tau^{-1} w' \end{pmatrix}. \label{eq:21}
\end{gather}
Stating the vector $Y(t) := (u(\cdot, t), \, v(\cdot, t), \, w(\cdot, t)) \in \mathcal{D}(\mathbb{A})$ for all $t \geqslant 0$, we recast \eqref{eq:pb1} as an abstract evolution problem in $\mathcal{H}$:
\begin{equation}   \label{eq:22}
	\left\{
	\begin{array}{>{\displaystyle}l}
	 \dfrac{dY(t)}{dt} = \mathbb{A} Y(t), \; t > 0 \\
	 Y(0) = Y_0 = \left(u_0, \, u_1, \, f_0(-s\tau)\right)
	\end{array}
	\right.
\end{equation}~

We now state the first main result of this paper. 

\begin{theorem} \label{t1}
Assume that the function $\sigma$ is (WD) or (SD). Under assumptions \eqref{eq:0} and \eqref{eq:19}, the linear operator $\mathbb{A}$ defined by \eqref{eq:20}-\eqref{eq:21} is densely defined in $\mathcal{H}$ and generates a $C_0-$semigroup of contractions. In particular:
\begin{itemize}
\item if $Y_0 \in \mathcal{D}\left(\mathbb{A}\right)$, then system \eqref{eq:22} admits a unique solution $Y \in C\left([0, \infty); \mathcal{D}\left(\mathbb{A}\right)\right) \cap C^1\left([0, \infty); \mathcal{H}\right)$. Precisely, we have:
\begin{gather}
\begin{split}  \label{eq:23}
u \in C^2\left([0, \infty); \, L^2(0,1) \right) \cap C^1\left([0, \infty); \, H^2_{\sigma,0}(0,1) \right) \cap C\left([0, \infty); \, Q_{\sigma,0}(0,1) \right);
\end{split}
\end{gather}
\item if $Y_0 \in \mathcal{H}$, then system \eqref{eq:22} admits a unique solution $Y \in C\left([0, \infty); \mathcal{H}\right)$. Specifically, we have:
\begin{gather}
\begin{split}  \label{eq:24}
u \in C^1\left([0, \infty); \, L^2(0,1) \right) \cap C\left([0, \infty); \, H^2_{\sigma,0}(0,1) \right).
\end{split}
\end{gather}
\end{itemize}
\end{theorem}

To demonstrate this theorem, we need the following lemma.
\begin{lemma} \label{l1}
Assume that the function $\sigma$ is (WD) or (SD). The operator $\mathbb{A}$ defined by \eqref{eq:20}-\eqref{eq:21} is $m-$dissipative. 
\end{lemma}

\begin{proof} ~\newline
\underline{$\mathbb{A}$ \itshape is dissipative}. Let $Y = (u, v, w) \in \mathcal{D}(\mathbb{A})$. We have:
\begin{align}
\begin{split} \label{eq:25}
\prec \mathbb{A} Y, Y\succ & = \int_0^1 \big( (qu')'-(\sigma u'')'' \big)(x) v(x) \, dx + \int_0^1 q(x)u'(x)v'(x) \, dx  \\
& \quad + \int_0^1 \sigma(x)v''(x)u''(x) \, dx - \gamma \int_0^1 w(s) w'(s) \, ds.
\end{split}
\end{align}
Integrating it by parts and using the boundary conditions $\eqref{eq:pb1}_3-\eqref{eq:pb1}_5$, we obtain:
\begin{align}
\prec \mathbb{A} Y, Y\succ & = -\left(\kappa_1 - \frac{\gamma}{2} \right) w^2(0) - \frac{\gamma}{2} w^2(1) - v'(1)^2 - \kappa_2 w(0) w(1). \label{eq:29}
\end{align}
Applying Young's inequality, \eqref{eq:29} leads to the following result:
\begin{gather} \label{eq:31}
	\prec \mathbb{A} Y, Y\succ \leqslant -\left(\kappa_1 - \frac{\gamma + \left|\kappa_2\right|}{2}\right) w^2(0) - \frac{\gamma - \left|\kappa_2\right|}{2} w^2(1) - v'(1)^2 \leqslant 0,
\end{gather}
due to inequalities \eqref{eq:0} and \eqref{eq:19}.

~\newline
\underline{$\mathbb{A}$ \itshape is maximal}. It is sufficient to prove that the operator $I-\mathbb{A}$ is surjective. \newline
Given $F = (f, g, h) \in \mathcal{H}$, we look for an element $Y = (u, v, w) \in \mathcal{D}\left( \mathbb{A} \right)$ such that $\left( I-\mathbb{A} \right) Y = F$. This consists of solving the following system:
\begin{gather} \label{eq:32}
\left\{
\begin{array}{l}
u - v = f \\
(\sigma u'')'' - (qu')' + u = f + g \\
w + \tau^{-1} w' = h.
\end{array}
\right.
\end{gather}
Solving $\eqref{eq:32}_1$ and $\eqref{eq:32}_3$ yields:
\begin{gather} \label{eq:33}
\left\{
\begin{array}{>\displaystyle l}
v = u - f \\
w(s) = \left(u(1) - f(1)\right)e^{-\tau s} + \tau \int_0^s e^{(r-s)\tau} h(r) dr, \quad s \in (0, 1).
\end{array}
\right.
\end{gather}
Consequently, we remark that the solution $Y$ is entirely determined by the knowledge of $u \in Q_{\sigma,0}(0,1)$. The latter solves:
\begin{gather} \label{eq:33a}
\left\{
\begin{array}{>\displaystyle l}
(\sigma u'')'' - (qu')' + u = f + g \\
(\sigma u'')(1) + u'(1) = f'(1) \\
(\sigma u'')'(1) - q(1)u'(1) = \Lambda_1 u(1) - \Lambda_2(f,h),
\end{array}
\right.
\end{gather}
where the constants $\Lambda_1$, $\Lambda_2(f,h)$ are given by:
\begin{gather}
\begin{split} \label{eq:38}
\Lambda_1 := \kappa_1 + \kappa_2 e^{-r}, \quad \Lambda_2(f,h) := \Lambda_1 f(1) - \kappa_2 \tau \int_0^1 e^{(r-1)\tau} h(r) \, dr.
\end{split}
\end{gather}
We adopt a variational approach to prove the existence of $u$. \newline
Let $\varphi \in H^2_{\sigma,0}(0,1)$. Multiplying $\eqref{eq:33a}_1$ by the test function $\varphi$ and integrating it over $[0,1]$, we get:
\begin{align}
	\int_0^1 \Big((\sigma u'')''(x) - (qu')'(x) + u(x)\Big) \varphi(x) \, dx = \int_0^1 \Big(f(x) + g(x)\Big) \varphi(x) \, dx. \label{eq:34}
\end{align}
Performing integrations by parts yields:
\begin{align} \label{eq:39}
\mathbb{B}_1(u, \varphi) = \mathbb{L}_1(\varphi), \quad \forall \varphi \in H^2_{\sigma,0}(0,1)
\end{align}
where the applications $\mathbb{B}_1$ and $\mathbb{L}_1$ are defined as follows:
\begin{align}
\mathbb{L}_1(\varphi) &:= \int_0^1 \Big(f(x) + g(x)\Big) \varphi(x) \, dx + \Lambda_2(f,h) \varphi(1) + f'(1) \varphi'(1), \label{eq:40}\\
\begin{split} \label{eq:41}
\mathbb{B}_1(u, \varphi) &:= \int_0^1 \Big(\sigma(x) u''(x) \varphi''(x) + q(x)u'(x)\varphi'(x) + u(x)\varphi(x)\Big) \, dx \\
& \quad + \Lambda_1 u(1) \varphi(1) + u'(1) \varphi'(1). 
\end{split}
\end{align}
It is straightforward to verify that the bilinear form $\mathbb{B}_1$ is continuous and coercive, and the linear form $\mathbb{L}_1$ is continuous. According to Lax-Milgram theorem, the variational problem \eqref{eq:39} admits a unique solution $u \in H_{\sigma,0}^2(0,1)$. Since $v = u - f$, it follows that $v \in H_{\sigma,0}^2(0,1)$.

Conversely, let us prove that $u \in H_{\sigma,0}^2(0,1)$ solves  $\eqref{eq:33a}$. First, \eqref{eq:39} holds for any $\varphi \in \mathcal{D}(0,1) \subset H_{\sigma,0}^2(0,1)$. Then, for any $\varphi \in \mathcal{D}(0,1)$, we have:
\begin{align} \label{eq:42a}
\int_0^1 \sigma(x) u''(x) \varphi''(x) \, dx + \int_0^1 q(x)u'(x)\varphi'(x) \, dx = \int_0^1 \Big(f(x) + g(x)- u(x)\Big) \varphi(x) \, dx. 
\end{align}
By integrating the first two integral terms of the preceding expression by parts, it follows that 
$\displaystyle (\sigma u'')'' - (qu')' + u = f+g $ a.e. in $(0,1)$. Since $\mathcal{D}(0,1)$ is dense in $L^2(0,1)$, then $(\sigma u'')'' - (qu')' + u = f+g$ on $L^2(0,1)$. Consequently, $\sigma u'' \in H^2(0,1)$. Thus \textcolor{blue}{$u \in Q_{\sigma}(0,1)$} and solves $\eqref{eq:33a}_1$. \newline
On the other hand, performing integrations by parts and 
using the relation \eqref{eq:42a}, we deduce from \eqref{eq:39} that: 
\begin{gather} 
\Big( q(1)u'(1) - (\sigma u'')'(1) + \Lambda_1 u(1) - \Lambda_2(f,h) \Big) \varphi(1) + \left( (\sigma u'')(1) + u'(1)- f'(1) \right) \varphi'(1) \nonumber \\
- \sigma(0)u''(0) \varphi'(0) \label{eq:43a}
= 0,
\end{gather}
for all $\varphi \in H_{\sigma,0}^2(0,1)$. It follows that:
\begin{gather} 
\label{eq:43b}
\left\{
\begin{array}{>\displaystyle l}
\big( \sigma u'' \big)(0) \varphi'(0) = 0 \\
\left( (\sigma u'')(1) + u'(1)- f'(1) \right) \varphi'(1) = 0 \\
\left( -(\sigma u'')'(1) + q(1)u'(1) + \Lambda_1 u(1) - \Lambda_2(f,h) \right) \varphi(1) = 0.
\end{array}
\right.
\end{gather}
As $\varphi$ is arbitrary, we deduce from relations $\eqref{eq:43b}_2$ and $\eqref{eq:43b}_3$ that $u$ satisfies the boundary conditions $\eqref{eq:33a}_2$ and $\eqref{eq:33a}_3$.
If $\sigma$ is (WD), then $u'(0) = \varphi'(0) = 0$ and the relation $\eqref{eq:43b}_1$ holds. Otherwise, if $\sigma$ is (SD), then $(\sigma u)''(0) = 0$. Finally, $u \in Q_{\sigma, 0}(0,1)$ and solves \eqref{eq:33a}.
We conclude that $I-\mathbb{A}$ is surjective.
\end{proof}

\begin{proof}[Proof of Theorem \ref{t1}] By Lemma \ref{l1}, the operator $\mathbb{A}$ is $m-$dissipative. Then it is densely defined in $\mathcal{H}$. According to L\"umer-Phillips Theorem,  $\mathbb{A}$ generates a $C_0-$semigroup of contractions. We obtain the desired result by using Hille-Yosida theorem. 
\end{proof}

	This section is devoted to the exponential stability of system \eqref{eq:p1}. To do so, inspired by the works presented in \cite{han2016, camasta2025, camasta2024, alabau2006, alabau2017, serge2006, serge2016, liao2024, salhi2025}, our approach is based on the Lyapunov method. \newline

In the following, we suppose that conditions \eqref{eq:0} and \eqref{eq:19} hold. To begin, we recall an important result by Kormonik \cite{komornik1994}.
	
\begin{theorem} \label{kormonik}
Suppose that $E: [0, +\infty) \to [0, +\infty)$ is a non-increasing function and that there exists a constant $M>0$ such that
\begin{gather}
\int_{t}^{\infty} E(s) \, ds \leqslant M E(t), \quad \forall t \in [0, +\infty).
\end{gather}
Then we have
\begin{gather}
E(t) \leqslant E(0) \,  e^{1 - \frac{t}{M}}, \quad \forall t \in [M, +\infty).
\end{gather}
\end{theorem}

	\subsection{Energy estimates}

Consistent with well established methodologies (see e.g. \cite{han2016, baysal2024, camasta2025, camasta2024, liao2024, alabau2006, alabau2017, akil2021, salhi2025}, the Lyapunov functional is derived from the energy of the system. In this subsection, we establish the energy of system \eqref{eq:p1} and prove its decay.
	
\begin{definition} \label{d1}
Assume that the function $\sigma$ is (WD) or (SD). \textcolor{blue}{The energy $E(t)$} of the delayed system \eqref{eq:p1}, associated with its solution $u$, is defined by:
\begin{align}  \label{eq:56}
\begin{split}
E(t)  := & \frac{1}{2} \left[ \int_0^1 \Big( u_t^2(x,t) + \sigma(x) u_{xx}^2(x,t) + q(x) u_x^2(x,t) \Big) \, dx \right. \\
& \qquad \left. + \gamma \tau \int_0^1 u_t^2(1,t - \tau s) \, ds \right], \quad \forall t \geqslant 0.
\end{split}
\end{align}
\end{definition}

\begin{proposition} \label{p2}
Assume that the function $\sigma$ is (WD) or (SD). Let $u$ be a regular solution of problem \eqref{eq:p1}. The energy $E(t)$ of system \eqref{eq:p1}, defined in \eqref{eq:56}, is dissipative. More precisely, we have:
\begin{align} \label{eq:57}
	\frac{d}{dt}E(t) \leqslant - C_{\kappa_1, \kappa_2}^\gamma \Big(u_t^2(1,t) + u_t^2(1,t-\tau) + u_{xt}^2(1,t) \Big), \quad \forall t \geqslant 0,
\end{align}
where the positive constant $C_{\kappa_1, \kappa_2}^\gamma$ is given by:
\begin{gather} \label{eq:58}
C_{\kappa_1, \kappa_2}^\gamma := \min \Big\{ 1; \, \frac{\gamma - \left|\kappa_2 \right|}{2}; \, \kappa_1 - \frac{\gamma + \left|\kappa_2 \right|}{2}\Big\}.
\end{gather}
\end{proposition}

\begin{proof}
We multiply equation $\eqref{eq:pb1}_1$ by $u_t$ and we integrate it over $(0,1)$ by parts. After using the boundary conditions $\eqref{eq:p1}_2$ and $\eqref{eq:p1}_3$, we obtain:
\begin{align} \label{eq:63}
\begin{gathered}
\frac{d}{dt} \left[ \frac{1}{2} \int_{0}^{1} \Big( u_t^2(x,t) + \sigma(x) u_{xx}^2(x,t) + q(x) u_{x}^2(x,t) \Big)\, dx \right] \qquad \\
+ \Big( (\sigma(x) u_{xx})_x - q(x) u_x \Big)(1,t) u_t(1,t)
 - (\sigma(x) u_{xx})(1,t) u_{xt}(1,t) = 0. 
\end{gathered}
\end{align}
Furthermore, multiplying $\eqref{eq:pb1}_2$ by $w$ and performing integration by parts yields:
\begin{gather}
\frac{d}{dt}\left( \frac{\tau}{2} \int_{0}^{1}  w^2(s,t) \, ds \right) + \frac{1}{2} \Big( w^2(1,t) - w^2(0,t) \Big) = 0. \label{eq:66}
\end{gather}
Adding \eqref{eq:63} and $\gamma$\eqref{eq:66} and incorporating the boundary conditions $\eqref{eq:pb1}_4-\eqref{eq:pb1}_7$, we obtain:
\begin{align}
\begin{split}
&\frac{d}{dt} \left[ \frac{1}{2} \int_{0}^{1} \Big( u_t^2 + \sigma(x) u_{xx}^2 + q(x) u_{x}^2 \Big)\, dx + \frac{\gamma \tau}{2} \int_{0}^{1}  u_t^2(1,t- \tau s) \, ds \right]  \\
& \quad = - \left( \kappa_1 - \frac{\gamma}{2} \right) u_t^2(1,t) - \frac{\gamma}{2} u_t^2(1,t-\tau) - u_{xt}^2(1,t) - \kappa_2 u_t(1,t) u_t(1,t- \tau).
\end{split}
\end{align}
 After using Young's inequality and identities \eqref{eq:0} and \eqref{eq:19} to deduce that:
\begin{align*}
 \frac{d}{dt} E(t) & \leqslant - \frac{\gamma - |\kappa_2|}{2} u_t^2(1,t- \tau) - \left( \kappa_1 - \frac{\gamma + |\kappa_2|}{2} \right) u_t^2(1,t) - u_{xt}^2(1,t) \\
 \frac{d}{dt} E(t)	& \leqslant - C_{\kappa_1, \kappa_2}^\gamma \Big( u_t^2(1,t-\tau) + u_t^2(1,t) + u_{xt}^2(1,t) \Big) \leqslant 0,
\end{align*}
where $\displaystyle C_{\kappa_1, \kappa_2}^\gamma > 0$ is defined in \eqref{eq:58}.
\end{proof}

	\subsection{Exponential decay of the energy}
 	
 		In this section, we prove the exponential decay of the energy of system \eqref{eq:p1} and provide a decay rate estimate. To this end, we construct a Lyapunov functional which decreases along the trajectories of delayed system \eqref{eq:p1}. \newline
		
		Consider the following Lyapunov functional:
\begin{gather} \label{eq:70}
L(t) := E(t) + \varepsilon G(t), \quad t \geq 0, 
\end{gather}
where $\varepsilon > 0$ is a sufficiently small constant that we will choose hereinafter, $E(t)$ is the energy defined in \eqref{eq:56} and the functional $G(t)$ is given by:
\begin{gather} \label{eq:71}
G(t) := \int_0^1 u_t(x,t) \Big( 2x u_x(x,t) + \frac{\iota_{\sigma, q}}{2} u(x,t) \Big) \, dx + \gamma \tau \int_0^1 e^{-2\tau s} u_t^2(1, t-\tau s) \, ds,
\end{gather}	
with the constants $\gamma$ satisfying \eqref{eq:19} and  $\iota_{\sigma, q}$ such that:
\begin{gather} \label{eq:72}
\iota_{\sigma, q} := \max\left\{ \iota_\sigma, \frac{q_2}{q_0}\right\} < 2.
\end{gather}

The following proposition establishes the equivalence between the Lyapunov function $L(t)$ and the energy $E(t)$.
\begin{proposition} \label{p3}
Suppose that the function $\sigma$ is (WD) or (SD). For $\varepsilon > 0$ small enough, there are two positive constants $\Theta_1, \, \Theta_2$ such that:
\begin{gather} \label{eq:73}
\Theta_1 E(t) \leqslant L(t) \leqslant \Theta_2 E(t), \quad \forall t \geqslant 0,
\end{gather}
where $\displaystyle \Theta_1 \text{ and } \, \Theta_2$ are given by:
\begin{gather} \label{eq:74}
	\Theta_1 := 1 - \varepsilon C_{\iota_{\sigma, q}}, \quad \Theta_2 := 1 + \varepsilon C_{\iota_{\sigma, q}}
\end{gather}
with the positive constant $C_{\iota_{\sigma, q}}$ defined by:
\begin{gather} 
C_{\iota_{\sigma, q}} := 2 \max \left\{ 1; \, 1 + \frac{\iota_{\sigma, q}}{4}, \, \frac{1}{q_0} \left(1 + \frac{\iota_{\sigma, q}}{8} \right) \right\}. \label{eq:78a}
\end{gather}
\end{proposition}

\begin{proof}
Applying Young's inequality, we get:
\begin{align} \label{eq:75}
\begin{split}
\left| \int_0^1 x u_t(x,t) u_x(x,t) dx \right|
& \leqslant \frac{1}{2} \left( \int_0^1 u_t^2(x,t) dx + \int_0^1 u_x^2(x,t) dx \right) \\
\left| \int_0^1 x u_t(x,t) u_x(x,t) dx \right| & \leqslant \frac{1}{2} \left( \lVert u_t(\cdot,t)\rVert_{L^2(0,1)}^2  + \frac{1}{q_0} \lVert \sqrt{q} u_x(\cdot,t)\rVert_{L^2(0,1)}^2  \right)
\end{split}
\end{align}
and
\begin{align} \label{eq:78}
\begin{split}
\left| \int_0^1 u_t(x,t) u(x,t) dx \right| 
	&\leqslant \frac{1}{2} \left( \int_0^1 u_t^2(x,t) dx + \int_0^1 u^2(x,t) dx \right) \\
\left| \int_0^1 u_t(x,t) u(x,t) dx \right| 	&\leqslant \frac{1}{2} \lVert u_t(\cdot,t)\rVert_{L^2(0,1)}^2 + \frac{1}{4 q_0} \lVert \sqrt{q} u_x(\cdot,t)\rVert_{L^2(0,1)}^2 .
\end{split}
\end{align}
So, using inequalities \eqref{eq:75} and \eqref{eq:78}, we have:
\begin{align*}
\left|G(t)\right|  \leqslant & \left( 1 + \frac{\iota_{\sigma, q}}{4}\right)\lVert u_t(\cdot,t)\rVert_{L^2(0,1)}^2 + \frac{1}{q_0} \left(1 + \frac{\iota_{\sigma, q}}{8} \right) \lVert \sqrt{q} u_x(\cdot,t)\rVert_{L^2(0,1)}^2 \\
& + \gamma \tau \int_0^1 u_t^2(1, t-\tau s) ds  \\
\left|G(t)\right| \leqslant & C_{\iota_{\sigma, q}} E(t),
\end{align*}
with the positive constant $C_{\iota_{\sigma, q}}$ defined by \eqref{eq:78a}.
By applying the triangle inequality on the preceding expression, we obtain the desired result. 
\end{proof}

As the energy derivative $E(t)$ is known, estimating the derivative of $L(t)$ requires that of the auxiliary term $G(t)$. This estimate is established in the result presented below.
\begin{proposition} \label{p4}
Let $\sigma$ be (WD) or (SD), and assume that \eqref{eq:72} holds. Then, for any regular solution $u$ of system \eqref{eq:p1}, we have the following estimate:
\begin{align} \label{eq:79}
\begin{split}
\frac{d}{dt} G(t) \leqslant 
& -\min\Big\{ 2 - \iota_{\sigma, q}; \, 4e^{-2\tau} \Big\} E(t) + \frac{9}{2} q(1) u_x^2(1,t) +  \frac{q(1)}{4} \iota_{\sigma, q}^2 u^2(1,t) \\
& + \left( 1 + \gamma + \frac{4}{q(1)} \kappa_1^2 \right) u_t^2(1,t) +  \left( \frac{4}{q(1)}\kappa_2^2-\gamma e^{-2\tau} \right)  u_t^2(1,t-\tau) \\
& + \left( \frac{1}{\sigma(1)} + \frac{1}{q(1)} \left( 2 + \frac{\iota_{\sigma, q}}{2} \right)^2 \right) u_{xt}^2(1,t)
, \quad \forall t \geqslant 0.
\end{split}
\end{align}
\end{proposition}

\begin{proof}
Firstly, we have:
\begin{gather} \label{eq:80}
\frac{d}{dt}\left( \int_0^1 x u_t(x,t)  u_x(x,t) dx \right) = \int_0^1 x \Big( u_{tt}(x,t) u_x(x,t) + u_t(x,t) u_{xt}(x,t)\Big) \, dx.
\end{gather}
Let us begin by estimating the first term on the right-hand side of the preceding identity. By equation $\eqref{eq:p1}_1$, we get:
\begin{gather}
\int_0^1 x u_{tt}(x,t) u_x(x,t) dx =  \int_0^1 x (q u_x)_x(x,t) u_x(x,t) dx - \int_0^1 x (\sigma u_{xx})_{xx}(x, t) u_x(x,t) dx. \label{eq:81}
\end{gather}
Performing integrations by parts and using the boundary conditions $\eqref{eq:p1}_2-\eqref{eq:p1}_4$, we obtain: 
\begin{align}
& \int_0^1 x (q u_x)_x(x,t) u_x(x,t) dx = \frac{1}{2} q(1) u_x^2(1,t) + \frac{1}{2} \int_0^1 \left( -q(x) + xq'(x) \right) u_x^2(x,t) dx \label{eq:82}\\
\begin{split}
& \int_0^1 x (\sigma(x) u_{xx})_{xx}(x, t) u_x(x,t) dx  = - \frac{1}{2} (\sigma(x) u_{xx}^2)(1,t) - (\sigma(x) u_{xx})(1,t) u_x(1,t)  \\
& \qquad + (\sigma u_{xx})_x(1,t) u_x(1,t) + \frac{1}{2}\int_0^1 \left( 3 \sigma(x) - x \sigma'(x) \right) u_{xx}^2(x,t) dx.
\end{split} \label{eq:85}
\end{align}
By the above identities \eqref{eq:82} and \eqref{eq:85}, the expression \eqref{eq:81} becomes:
\begin{align} \label{eq:86}
\begin{split}
\int_0^1 x u_{tt}(x,t) & u_x(x,t) dx  =  \frac{1}{2} q(1) u_x^2(1,t) + \frac{1}{2} (\sigma(x) u_{xx}^2)(1,t) + (\sigma(x) u_{xx})(1,t) u_x(1,t) \\
& -(\sigma(x) u_{xx})_x(1,t) u_x(1,t) + \frac{1}{2}\int_0^1 \left( -3 \sigma(x) + x \sigma'(x) \right) u_{xx}^2(x,t) \, dx\\
& +\frac{1}{2} \int_0^1 \left( -q(x) + xq'(x) \right) u_x^2(x,t) dx.
\end{split}
\end{align}
In addition, we have:
\begin{align} \label{eq:87}
\begin{split}
\int_0^1 x u_t(x,t) u_{xt}(x,t) dx & = \frac{1}{2} u_t^2(1,t) - \frac{1}{2} \int_0^1 u_t^2(x,t) dx.
\end{split}
\end{align}
Overall, by using \eqref{eq:86} and \eqref{eq:87}, the following holds:
\begin{align} \label{eq:88}
\begin{split}
& \frac{d}{dt}\left( \int_0^1 x u_t(x,t)  u_x(x,t) \, dx \right) \\
& = \frac{1}{2} \int_0^1 \Big(-q(x) + xq'(x) \Big) u_x^2(x,t) \, dx + \frac{1}{2}\int_0^1 \Big( -3 \sigma(x) + x \sigma'(x) \Big) u_{xx}^2(x,t) \, dx \\
& \quad - \frac{1}{2} \int_0^1 u_t^2(x,t) \, dx  + \frac{1}{2} u_t^2(1,t) + \frac{1}{2} q(1) u_x^2(1,t) + \frac{1}{2} \Big( \sigma(x) u_{xx}^2 \Big)(1,t) \\
& \quad + \Big( \sigma(x) u_{xx} \Big)(1,t) u_x(1,t) - \Big( \sigma(x) u_{xx}\Big)_x(1,t) u_x(1,t).
\end{split}
\end{align}
Secondly, we have:
\begin{align} \label{eq:89}
\begin{split}
\frac{d}{dt} \left( \int_0^1 u_t(x,t) u(x,t) \, dx \right) = & \int_0^1 \Big( \big( q(x) u_x \big)_x - \big( \sigma(x) u_{xx} \big)_{xx} \Big)(x,t) u(x,t)  \, dx \\
& + \int_0^1  u_t^2(x,t) \, dx.
\end{split}
\end{align}
After using integrations by parts, it follows that:
\begin{align} \label{eq:90}
\int_0^1 \big( q(x)u_x \big)_x(x,t) u(x,t) \, dx = & \, q(1) u_x(1,t) u(1,t) - \int_0^1 q(x) u_x^2(x,t) \, dx;\\
\begin{split} \label{eq:90a}
	\int_0^1 \big( \sigma(x) u_{xx} \big)_{xx}(x,t) u(x,t) \, dx = & - \big( \sigma(x) u_{xx} \big)(1,t) u_x(1,t) + \big( \sigma(x) u_{xx} \big)_x(1,t) u(1,t) \\
	& + \int_0^1 \sigma(x) u_{xx}^2(x,t) \, dx.
\end{split}
\end{align}
Consequently, we get:
\begin{align} \label{eq:92}
\begin{split}
& \frac{d}{dt} \left( \int_0^1 u_t(x,t) u(x,t) \, dx \right) = \int_0^1 u_t^2(x,t) \, dx - \int_0^1 q(x) u_x^2(x,t) - \int_0^1 \sigma(x) u_{xx}^2(x,t) \, dx \\
& \quad +  \Big( q(1) u_x(1,t) - (\sigma(x) u_{xx})_x(1,t) \Big) u(1,t) + (\sigma(x) u_{xx})(1,t) u_x(1,t).
\end{split}
\end{align}
Finally, by using the identities \eqref{eq:88} and \eqref{eq:92}, we obtain:
\begin{align} \label{eq:93}
\begin{split}
& \frac{d}{dt} \left( \int_0^1 u_t(x,t) \Big( 2x u_x(x,t) + \frac{\iota_{\sigma, q}}{2} u(x,t) \Big) \, dx \right) \\
& = \left( -1 + \frac{\iota_{\sigma, q}}{2} \right) \int_0^1 u_t^2(x,t) \, dx + \int_0^1 \Big( -\left(3 + \frac{\iota_{\sigma, q}}{2} \right) \sigma(x) + x \sigma'(x) \Big) u_{xx}^2(x,t) \, dx \\
& \quad + \int_0^1 \left( -\Big( 1+\frac{\iota_{\sigma, q}}{2} \Big) q(x) + x q'(x) \right) u_x^2(x,t) \, dx + u_t^2(1,t) + \sigma(1) u_{xx}^2(1,t) \\
& \quad + \Big( q(1) u_x(1,t) - (\sigma u_{xx})_x(1,t) \Big) \left( \frac{\iota_{\sigma, q}}{2} u(1,t) + u_x(1,t) \right) \\
& \quad + \left( 2 + \frac{\iota_{\sigma, q}}{2} \right) (\sigma(x) u_{xx})(1,t) u_x(1,t) - (\sigma u_{xx})_x(1,t) u_x(1,t)
\end{split}
\end{align}
Next, using the definition of $\displaystyle \iota_{\sigma, q}$, we have:
\begin{gather} \label{eq:94}
\begin{split}
-1 + \frac{\iota_{\sigma, q}}{2} \leqslant 0, \quad -\Big( 1+\frac{\iota_{\sigma, q}}{2} \Big) q(x) + x q'(x) \leqslant \Big( -1 + \frac{\iota_{\sigma, q}}{2} \Big) q(x) \leqslant 0,  \\
- \left( 3+\frac{\iota_{\sigma, q}}{2} \right) \sigma(x) + x \sigma'(x) \leqslant \left( -1 + \frac{\iota_{\sigma, q}}{2} \right) \sigma(x) \leqslant 0.
\end{split}
\end{gather}
So, we obtain:
\begin{align} \label{eq:95}
\begin{split}
& \frac{d}{dt} \left( \int_0^1 u_t(x,t) \Big( 2x u_x(x,t) + \frac{\iota_{\sigma, q}}{2} u(x,t) \Big) dx \right) \leqslant u_t^2(1,t) + \sigma(1) u_{xx}^2(1,t) \\
& \quad + \Big( -1 + \frac{\iota_{\sigma, q}}{2} \Big) \left( \int_0^1 u_t^2(x,t) \, dx + \int_0^1 q(x) u_x^2(x,t) \, dx + \int_0^1 \sigma(x) u_{xx}^2(x,t) \, dx \right)  \\
& \quad \underbrace{- \big(\sigma(x) u_{xx} \big)_x(1,t) u_x(1,t) + \left( 2 + \frac{\iota_{\sigma, q}}{2} \right) \big( \sigma(x) u_{xx} \big)(1,t) u_x(1,t)}_{(i)} \\
& \quad + \underbrace{\Big( q(1) u_x(1,t) - \big( \sigma(x) u_{xx} \big)_x(1,t) \Big) \Big( \frac{\iota_{\sigma, q}}{2} u(1,t) + u_x(1,t) \Big)}_{(ii)}.
\end{split}
\end{align}
The Young's inequality applied to $(i)$ and $(ii)$, together with boundary condition $\eqref{eq:p1}_4$ yields:
\begin{align}
\begin{split}
(i) \leqslant  
& \, \frac{7}{2} q(1)u_x^2(1,t) + \frac{3}{q(1)} \Big( \kappa_1^2 u_t^2(1,t) + \kappa_2^2 u_t^2(1, t-\tau)\Big) + \frac{1}{q(1)} \left( 2 + \frac{\iota_{\sigma, q}}{2} \right)^2 u_{xt}^2(1,t)
\end{split} \\
(ii) & \leqslant \frac{1}{q(1)} \Big( \kappa_1^2 u_t^2(1,t) + \kappa_2^2 u_t^2(1, t-\tau) \Big) + q(1)\frac{\iota_{\sigma, q}^2}{4}  u^2(1,t) + q(1) u_x^2(1,t).
\end{align}
Then \eqref{eq:95} becomes:
\begin{align} \label{eq:98}
\begin{split}
& \frac{d}{dt} \left( \int_0^1 u_t(x,t) \Big( 2x u_x(x,t) + \frac{\iota_{\sigma, q}}{2} u(x,t) \Big) dx \right) \\
& \quad \leqslant \Big( -1 + \frac{\iota_{\sigma, q}}{2} \Big) \left( \int_0^1 u_t^2(x,t) \, dx + \int_0^1 q(x) u_x^2(x,t) \, dx + \int_0^1 \sigma(x) u_{xx}^2(x,t) \, dx \right) \\
& \qquad + \left( 1 + \frac{4}{q(1)} \kappa_1^2 \right) u_t^2(1,t) + \frac{4}{q(1)} \kappa_2^2 u_t^2(1,t-\tau) + \frac{9}{2} q(1) u_x^2(1,t) + \frac{q(1)}{4} \iota_{\sigma, q}^2 u^2(1,t) \\
& \qquad + \left( \frac{1}{\sigma(1)} + \frac{1}{q(1)} \left( 2 + \frac{\iota_{\sigma, q}}{2} \right)^2 \right) u_{xt}^2(1,t). 
\end{split}
\end{align}
Furthermore, we have the following identity:
\begin{align} 
\begin{split} \label{eq:99}
\frac{d}{dt} \left( \int_0^1 e^{-2\tau s} u_t^2(1, t-\tau s) ds \right) & = \tau^{-1} \left( u_t^2(1,t) - e^{-2\tau} u_t^2(1, t-\tau) \right) \\
& \quad - 2 \int_0^1 e^{-2\tau s} u_t^2(1,t-\tau s) ds.
\end{split}
\end{align}
Finally, combining \eqref{eq:99} with \eqref{eq:98} yields the desired result.
\end{proof}

The following lemma establishes an integral energy estimate, a key result for proving the exponential stability of system \eqref{eq:p1}.
\begin{lemma} \label{l2}
Let $T>0$. Assume that \eqref{eq:72} holds and the function $\sigma$ is (WD) or (SD). For $\varepsilon > 0$ small enough,
we have for any $r \in (0,T)$:
\begin{align} \label{eq:102}
\begin{split}
\varepsilon \min\Big\{ 2 - \iota_{\sigma, q}; \, 4e^{-2\tau} \Big\} \int_r^T E(t) \, dt 
& \leqslant L(r) - L(T) \\
& \quad + \varepsilon C_0 \left[ \int_r^T u^2(1,t) \, dt + \int_r^T u_x^2(1,t) \, dt \right],
\end{split}
\end{align}
where $C_0$ is a positive constant given by:
\begin{align} \label{eq:103}
\begin{split}
C_0 &:= \max \left\{\frac{9}{2}q(1); \, \frac{q(1)}{4} \iota_{\sigma, q}^2 \right\}.
\end{split}
\end{align}
\end{lemma}

\begin{proof}
Using Propositions \ref{p2} and \ref{p4}, we obtain, for a sufficiently small $\varepsilon$:
\begin{align} \label{eq:104}
	\frac{d}{dt} L(t) & \leqslant  - C_{\kappa_1, \kappa_2}^\gamma \Big( u_t^2(1, t) + u_t^2(1, t-\tau) + u_{xt}^2(1, t) \Big) + \varepsilon \left[ -\min\Big\{ 2 - \iota_{\sigma, q}; \, 4e^{-2\tau} \Big\} E(t)  \right. \nonumber \\
& \left. \quad + \frac{9}{2} q(1) u_x^2(1,t) +  \frac{q(1)}{4} \iota_{\sigma, q}^2 u^2(1,t) + \left( 1 + \gamma + \frac{4}{q(1)} \kappa_1^2 \right) u_t^2(1,t)  \right. \nonumber\\
& \left. \quad +  \left( \frac{4}{q(1)}\kappa_2^2-\gamma e^{-2\tau} \right)  u_t^2(1,t-\tau) + \left( \frac{1}{\sigma(1)} + \frac{1}{q(1)} \left( 2 + \frac{\iota_{\sigma, q}}{2} \right)^2 \right) u_{xt}^2(1,t) \right] \nonumber \\
 		&\leqslant - \varepsilon \min\Big\{ 2 - \iota_{\sigma, q}; \, 4e^{-2\tau} \Big\} E(t) - C_\varepsilon \Big(u_t^2(1, t) + u_t^2(1, t-\tau) + u_{xt}^2(1,t)\Big) \nonumber \\
 		& \quad + \varepsilon C_0 \Big( u^2(1,t) + u_x^2(1,t) \Big) \nonumber \\
	\frac{d}{dt} L(t) & \leqslant - \varepsilon \min\Big\{ 2 - \iota_{\sigma, q}; \, 4e^{-2\tau} \Big\} E(t) + \varepsilon C_0 \Big( u^2(1,t) + u_x^2(1,t) \Big).
\end{align}
with the constants $C_\varepsilon$ and $C_0$ defined by:
\begin{align}
\begin{split}
C_\varepsilon &:= \min \left\{C_{\kappa_1, \kappa_2}^\gamma - \varepsilon  \frac{4}{q(1)} \kappa_2^2, \, C_{\kappa_1, \kappa_2}^\gamma - \varepsilon \left( 1 + \gamma + \frac{4}{q(1)} \kappa_1^2 \right), \right. \\
& \qquad \qquad \left. C_{\kappa_1, \kappa_2}^\gamma - \varepsilon \left( \frac{1}{\sigma(1)} + \frac{1}{q(1)} \left( 2 + \frac{\iota_{\sigma,q}}{2} \right)^2 \right)\right\}
\end{split}
\end{align} 
and \eqref{eq:103}, respectively. We choose $\varepsilon$ small enough such that:
\begin{gather} \label{eq:101}
\varepsilon < \min\left\{ \frac{C_{\kappa_1, \kappa_2}^\gamma }{4 \kappa_2^2}q(1); \; \frac{C_{\kappa_1, \kappa_2}^\gamma }{1 + \gamma + \frac{4}{q(1)} \kappa_1^2}, \, \frac{C_{\kappa_1, \kappa_2}^\gamma}{\frac{1}{\sigma(1)} + \frac{1}{q(1)} \left( 2 + \frac{\iota_{\sigma,q}}{2} \right)^2} \right\}.
\end{gather}
Then $C_0$ and $C_\varepsilon$ are positive. By integrating the differential inequality \eqref{eq:104} over $(r, T)$ for arbitrary $r \in (0, T)$, we deduce \eqref{eq:102}.
\end{proof}

To estimate the integral terms on the right-hand side of \textcolor{blue}{relation} \eqref{eq:102}, the following proposition proves the existence, uniqueness, and requisite critical bounds for the solution of a degenerate elliptic equation.
\begin{proposition} \label{p5}
 Suppose that the function $\sigma$ is (WD) or (SD). Define:
\begin{gather} \label{eq:105}
||| y |||^2 := \int_0^1 \sigma(x) (y''(x))^2 dx + \int_0^1 q(x) (y'(x))^2 dx, 
\end{gather}
for all $y \in H^2_{\sigma,0}(0,1)$. 
The norms $|||\cdot|||, \text{ and } \lVert \, \cdot  \,\lVert_{2, \sigma}$ are equivalent on $H^2_{\sigma,0}(0,1)$. In addition, for all $\lambda, \mu \in \mathbb{R}$, the variational problem 
\begin{gather} \label{eq:106}
\int_0^1 \sigma(x) y''(x) \varphi''(x) \, dx + \int_0^1 q(x) y'(x) \varphi'(x) \, dx = \lambda \varphi(1) + \mu \varphi'(1) \quad \forall \varphi \in H^2_{\sigma,0}(0,1),
\end{gather}
admits a unique solution $y \in H^2_{\sigma,0}(0,1)$, which satisfies the following estimates:
\begin{gather} \label{eq:107}
\lVert y \rVert^2_{L^2(0,1)} \leqslant  \frac{1}{q_0} C_{\sigma, q, \iota, \lambda, \mu}^2 \quad \text{ and } \quad ||| y |||^2 \leqslant  C_{\sigma, q, \iota, \lambda, \mu}^2 ,
\end{gather}
where the constant $C_{\sigma, q, \iota, \lambda, \mu}$ is defined by:
\begin{gather} \label{eq:108}
C_{\sigma, q, \iota, \lambda, \mu} := |\lambda| \sqrt{\frac{1}{q_0}} + \sqrt{2}|\mu| C_1, \text{ with } \, C_1 := \sqrt{\max \left\{ \frac{1}{q_0}; \, \frac{1}{\sigma(1) \left( 2-\iota_\sigma \right)} \right\}} .
\end{gather}
Moreover $y \in \mathcal{D}\left( A_\sigma \right) := Q_{\sigma, 0}(0, 1)$ verifies the following system:
\begin{gather} \label{eq:109}
\left\{
\begin{array}{l}
A_\sigma y = 0, \\
q(1) y'(1) - (\sigma y'')'(1) = \lambda, \\
\sigma(1) y''(1) = \mu.
\end{array}
\right. \; \text{ where } \quad
A_\sigma y := \left( \sigma y'' \right)'' - (q y')'.
\end{gather}
\end{proposition}

\begin{proof}
Let $\varphi \in H_{\sigma,0}^2(0,1)$. We have:
\begin{gather} \label{eq:110}
\int_{0}^{1} \Big( \left( \sigma y'' \right)'' - \left( q y' \right)'\Big)\varphi \, dx = 0.
\end{gather}
Integrating the left hand side of \eqref{eq:110} by parts and using the boundary conditions $\eqref{eq:109}_2-\eqref{eq:109}_3$, it follows that:
\begin{gather} \label{eq:111}
\underbrace{\int_{0}^{1} \sigma(x) y''(x) \varphi''(x) \, dx + \int_{0}^{1} q(x) y'(x) \varphi'(x) \, dx}_{\chi (y, \varphi)} = \underbrace{\lambda \varphi(1) + \mu \varphi'(1)}_{\Psi(\varphi)}.
\end{gather}
The variational problem associated with \eqref{eq:109} consists to find $y \in H_{\sigma,0}^2(0,1)$ satisfying:
\begin{gather} \label{eq:112}
\chi (y, \varphi) = \Psi(\varphi), \quad \forall \varphi \in H_{\sigma,0}^2(0,1).
\end{gather}

It is straightforward to prove that the bilinear form $\chi$ is continuous and coercive on $H_{\sigma,0}^2(0,1)$. Routine calculations show that the linear form $\Psi$ is continuous on $H_{\sigma,0}^2(0,1)$. According to Lax-Milgram theorem, there exists a unique solution $y \in H_{\sigma,0}^2(0,1)$ which solves \eqref{eq:112}. In particular, we have:
\begin{gather} \label{eq:113}
|||y|||^2 = \int_0^1 \Big( \sigma(x)(y''(x))^2 + q(x) (y'(x))^2 \Big) \, dx = \lambda y(1) + \mu y'(1).
\end{gather}
In addition, the following inequalities hold:
\begin{align} \label{eq:114}
\left( y(1) \right)^2  & \leqslant q_0^{-1} \lVert \sqrt{q}y' \rVert_{L^2(0,1)}^2 \leqslant q_0^{-1} |||y|||^2 \\
\begin{split} \label{eq:115}
\left( y'(1) \right)^2 & \leqslant 2 \left( \frac{1}{\sigma(1) \left( 2-\iota_\sigma \right)} \lVert \sqrt{\sigma}y'' \rVert_{L^2(0,1)}^2 + \frac{1}{q_0} \lVert \sqrt{q}y' \rVert_{L^2(0,1)}^2 \right) \\
\left( y'(1) \right)^2 & \leqslant 2 \max \left\{ \frac{1}{\sigma(1) \left( 2-\iota_\sigma \right)}; \, \frac{1}{q_0} \right\} |||y|||^2.
\end{split}
\end{align}
Then, we get:
\begin{align} \label{eq:116}
|||y|||^2 & \leqslant \left( |\lambda| \sqrt{\frac{1}{q_0}} + |\mu| \sqrt{2 \max \left\{ \frac{1}{q_0}; \, \frac{1}{\sigma(1) \left( 2-\iota_\sigma \right)} \right\}} \right) |||y||| 
\end{align}
Hence the second inequality of \eqref{eq:107} holds with the constant $C_{\sigma, q, \iota, \lambda, \mu}$ defined in \eqref{eq:108}.
Using \eqref{eq:10}, we also have:
\begin{gather} \label{eq:118}
\lVert y \rVert^2_{L^2(0,1)} \leqslant \frac{1}{q_0} |||y|||^2 \leqslant \frac{1}{q_0} C_{\sigma, q, \iota, \lambda, \mu}^2.
\end{gather}

Conversely, let $y$ be a weak solution of \eqref{eq:109}. Taking  $\varphi \in \mathcal{D}(0,1)$ and performing integrations by parts, we obtain:
\begin{align} \label{eq:120}
\begin{split}
\int_{0}^{1} \Big( (\sigma y'')''(x) - (qy')'(x) \Big) \varphi(x) \, dx = 0.
\end{split}
\end{align}
Then $\displaystyle (\sigma y'')'' - (q y')' = 0$ a.e. in $(0,1)$. As $\mathcal{D}(0,1)$ is dense in $L^2(0,1)$, equation $\eqref{eq:109}_1$ holds. Therefore $\sigma y'' \in H^2(0,1)$ and we deduce that $y \in Q_{\sigma}(0,1)$. \newline
Coming back to $H^2_{\sigma,0}(0,1)$ and after using integrations by parts, we obtain from the above that equation \eqref{eq:112} becomes:
\begin{itemize}
\item if $\sigma$ is (WD), then $y'(0) = \varphi'(0) = 0$ and:
\begin{gather} \label{eq:121}
\Big( q(1)y'(1) - (\sigma y'')'(1) - \lambda \Big) \varphi(1) + \Big( \sigma(1) y''(1) - \mu \Big) \varphi'(1) = 0;
\end{gather}
\item if $\sigma$ is (SD), then we have:
\begin{gather} \label{eq:122}
\begin{gathered} 
\Big( q(1)y'(1) - (\sigma y'')'(1) - \lambda \Big) \varphi(1) + \Big( \sigma(1) y''(1) - \mu \Big) \varphi'(1) \\
- \, \sigma(0) y''(0) \varphi'(0) = 0,
\end{gathered}
\end{gather}
\end{itemize}
for every $\varphi \in H^2_{\sigma,0}(0,1)$. Then the conditions $\eqref{eq:109}_2$ and $\eqref{eq:109}_3$ are satisfied regardless of the degeneracy type of the function $\sigma$ and $(\sigma y'')(0) = 0$ in the single case (SD).
Thus $y \in Q_{\sigma,0}(0,1)$. Finally, $y$ solves \eqref{eq:109}.
\end{proof}

By virtue of Proposition \ref{p5}, the following result provides a crucial energy-based estimate for the integral terms on the right-hand side of estimate \eqref{eq:102}.
\begin{lemma} \label{l3}
Suppose that the function $\sigma$ is (WD) or (SD). Let $u$ be a solution of system \eqref{eq:p1}. Under assumption \eqref{eq:72}, the following estimate holds:
\begin{align}  \label{eq:124}
\begin{split}
\int_r^T \left( u^2(1,t) + u_x^2(1,t) \right) \, dt & \leqslant 2 \bigg[ \tilde{\delta} \int_r^T E(t) \, dt + C_2^\delta \Big( E(r) - E(T) \Big) \\
& \qquad \; + C_3 \Big( E(r) + E(T) \Big) \bigg],
\end{split}
\end{align}
for every $\tilde{\delta} > 0$, where the constants $C_2^\delta$ and $C_3$ are defined as follows:
\begin{gather}
C_2^\delta 
:= \dfrac{ \max\left\{ 2C_1^2; \, \dfrac{1}{q_0} \right\} }{\tilde{\delta} q_0  C_{\kappa_1, \kappa_2}^\gamma} + \dfrac{\max \left\{ \dfrac{\kappa_1^2}{\delta}; \, \dfrac{\kappa_2^2}{\delta}; \, \dfrac{1}{2\delta} \right\}}{C_{\kappa_1, \kappa_2}^\gamma},  \label{eq:125}
\end{gather}
with
\begin{gather} \label{eq:125a}
\delta := \frac{1}{2} \left[ \max\left\{ 2C_1^2; \, \frac{1}{q_0} \right\}  \left( 2 C_1^2 + \frac{1}{q_0} \right) \right]^{-1}
\end{gather}
and 
\begin{gather} 
C_3 := \max\left\{1; \, \dfrac{2}{q_0^2} \left( 4 C_1^2 + \dfrac{1}{q_0} \right); \, \dfrac{8 C_1^2}{q_0 \sigma(1) \left( 2 - \iota_\sigma \right)} \right\}. \label{eq:126}
\end{gather}
\end{lemma}

\begin{proof}
Setting $\lambda := u(1,t)$ and $\mu = u_{x}(1,t)$, we consider $y(\cdot, t) \in H^2_{\sigma,0}(0,1)$ the unique solution of
\eqref{eq:106}. By Proposition \ref{p5}, $y(\cdot, t) \in \mathcal{D}(A_\sigma)$ and solves:
\begin{gather} \label{eq:127}
\left\{
\begin{array}{l}
\left( \sigma(x) y_{xx} \right)_{xx} - (q(x) y_x)_x = 0, \\
q(1) y_x(1,t) - (\sigma(x) y_{xx})_x(1,t) = \lambda, \\
\sigma(1) y_{xx}(1,t) = \mu.
\end{array}
\right.
\end{gather}
First, multiplying $\eqref{eq:p1}_1$ by $y$ and integrating it over $(r,T)\times(0,1)$, we have:
\begin{gather} \label{eq:128}
\int_r^T \int_0^1 \Big( u_{tt} + \left( \sigma(x) u_{xx} \right)_{xx} - \left( q(x) u_x \right)_x \Big) y \, dx \, dt = 0.
\end{gather}
Performing integrations by parts and using the boundary condition $\eqref{eq:p1}_3$, we obtain:
\begin{gather}
\begin{gathered} \label{eq:129}
\int_r^T \int_0^1 u_{tt} y \, dx \, dt = \left[ \int_0^1 u_t y dx \right]_{t=r}^{t=T} - \int_r^T \int_0^1 u_{t} y_t \, dx \, dt
\end{gathered} \\
\begin{gathered} \label{eq:130}
\int_r^T \int_0^1 \left( \sigma(x) u_{xx} \right)_{xx} y \, dx \, dt 
=  \int_r^T \int_0^1 \sigma(x) u_{xx} y_{xx} \, dx \, dt \\
\qquad + \int_r^T \left[ \big( \sigma(x) u_{xx} \big)_x(1,t) y(1,t) - \big( \sigma(x) u_{xx} \big)(1,t) y_x(1,t) \right] \, dt
\end{gathered} \\
\begin{gathered} \label{eq:131}
\int_r^T \int_0^1 \left( q(x) u_{x} \right)_{x} y \, dx \, dt = \int_r^T q(1) u_x(1,t) y(1,t) \, dt - \int_r^T \int_0^1 q(x) u_{x} y_x \, dx \, dt.
\end{gathered}
\end{gather}
Using identities \eqref{eq:129}-\eqref{eq:131}, 
the expression \eqref{eq:128} becomes:
\begin{align} \label{eq:132}
\begin{gathered}
\left[ \int_0^1 u_t y \, dx \right]_{t=r}^{t=T} 
- \int_r^T \int_0^1 u_{t} y_t \, dx \, dt  \\
+  \int_r^T \int_0^1 \sigma(x) u_{xx} y_{xx} \, dx \, dt + \int_r^T \int_0^1 q(x) u_{x} y_x \, dx \, dt \\
+ \int_r^T \left[ \Big( \left( \sigma(x) u_{xx} \right)_x(1,t) - q(1) u_x(1,t) \Big) y(1,t) - \left( \sigma(x) u_{xx} \right)(1,t) y_x(1,t) \right]\, dt  = 0. 
\end{gathered}
\end{align}
On the other part, multiplying $\eqref{eq:109}_1$ by $u$ and integrating it over $(r,T)\times(0,1)$, we have:
\begin{gather} \label{eq:133}
\int_r^T \int_0^1 \Big( \left( \sigma(x) y_{xx} \right)_{xx} - (q(x) y_x)_x \Big) u \, dx \, dt = 0.
\end{gather}
Upon integrating by parts, the preceding expression yields:
\begin{gather} 
\int_r^T \int_0^1 \Big( \sigma(x) y_{xx} u_{xx} + q(x) u_x y_x \Big) \, dx \, dt \,  \nonumber \\ 
= \int_r^T \left[ \Big( \left( q(1) y_x(1,t) - \sigma(x) y_{xx} \right)_{x}(1,t) \Big) u(1,t) + \sigma(1) y_{xx}(1,t) u_x(1,t) \right] \, dt. \label{eq:134}
\end{gather}
After incorporating the boundary conditions $\eqref{eq:p1}_3-\eqref{eq:p1}_4$ and substituting \eqref{eq:134} into \eqref{eq:132}, it follows that:
\begin{align} \label{eq:138}
\begin{split}
& \int_r^T \int_0^1 u_{t} y_t \, dx \, dt - \left[ \int_0^1 u_t y \, dx \right]_{t=r}^{t=T} 
= \int_r^T \left( u^2(1,t) + u_x^2(1,t) \right) \, dt \\
& \quad + \int_r^T \left[ \Big( \kappa_1 u_t(1,t) + \kappa_2  u_t(1,t-\tau) \Big) y(1,t) + u_{xt}(1,t) y_x(1,t) \right] \, dt.
\end{split}
\end{align}
Therefore:
\begin{align} \label{eq:139}
& \int_r^T \left( u^2(1,t) + u_x^2(1,t) \right) \, dt = \int_r^T \int_0^1 u_{t} y_t \, dx \, dt - \left[ \int_0^1 u_t y \, dx \right]_{t=r}^{t=T} \nonumber \\
& \qquad - \int_r^T \left[ \Big( \kappa_1 u_t(1,t) + \kappa_2  u_t(1,t-\tau) \Big) y(1,t) + u_{xt}(1,t) y_x(1,t) \right] \, dt.
\end{align}
Moreover, using the inequalities \eqref{eq:10}, \eqref{eq:11} and \eqref{eq:107}, we have:
\begin{align}
\left| \int_0^1 u_t(x,t) y(x,t) \, dx \right| & \leqslant \frac{1}{2} \left( \int_0^1  u_t^2(x,t) \, dx + \frac{1}{q_0} C_{\sigma, q, \iota, \lambda, \mu}^2 \right) \nonumber\\
& \quad \leqslant \frac{1}{2} \int_0^1 u_t^2(x,t) \, dx + \frac{1}{q_0^2} u^2(1,t) + \frac{2 C_1^2}{q_0} u_x^2(1,t) \nonumber \\
& \quad \leqslant \frac{1}{2} \int_0^1 u_t^2(x,t) \, dx + \left( \dfrac{1}{q_0^3} + \dfrac{4 C_1^2}{q_0^2} \right) \int_0^1 q(x) u_x^2 \, dx \nonumber \\
& \qquad + \frac{4 C_1^2}{q_0 \sigma(1) \left( 2 - \iota_\sigma \right)} \int_0^1 \sigma(x) u_{xx}^2 \, dx \nonumber \\
\left| \int_0^1 u_t(x,t) y(x,t) \, dx \right| & \leqslant \underbrace{\max\left\{1; \, \dfrac{2}{q_0^2} \left( 4 C_1^2 + \dfrac{1}{q_0} \right); \, \dfrac{8 C_1^2}{ q_0 \sigma(1) \left( 2 - \iota_\sigma \right)} \right\}}_{C_3}  E(t).
\end{align}
So, we get:
\begin{gather} \label{eq:140}
\left| \left[ \int_0^1 u_t(x,t) y(x,t) \, dx \right]_{t=r}^{t=T} \right| \leqslant C_3 \Big( E(T) + E(r) \Big).
\end{gather}
Applying $\delta-$Young's inequality, we obtain:
\begin{align}  \label{eq:141}
\begin{gathered}
\left| \int_r^T \left[ \Big( \kappa_1 u_t(1,t) + \kappa_2  u_t(1,t-\tau) \Big) y(1,t) + u_{xt}(1,t) y_x(1,t) \right] dt \right| 
\leqslant \frac{1}{2\delta} \int_r^T u_{xt}^2(1,t) \, dt \\ 
 + \frac{1}{\delta} \int_r^T \Big( \kappa_1^2 u_t^2(1,t) + \kappa_2^2  u_t^2(1,t-\tau) \Big) \, dt + \frac{\delta}{2} \int_r^T \left( y^2(1,t) + y_x^2(1,t) \right) \, dt.
\end{gathered}
\end{align}
Using the following inequality:
\begin{align} \label{eq:142}
\begin{split}
y^2(1,t) + y_x^2(1,t) \leqslant 2 \max\left\{ 2C_1^2; \, \frac{1}{q_0} \right\}  \left( 2 C_1^2 + \frac{1}{q_0}  \right) \left( u^2(1,t) +  u^2_x(1,t) \right),
\end{split}
\end{align}
the expression \eqref{eq:141} becomes:
\begin{align} \label{eq:143}
& \left| \int_r^T \left[ \Big( \kappa_1 u_t(1,t) + \kappa_2  u_t(1,t-\tau) \Big) y(1,t) + u_{xt}(1,t) y_x(1,t) \right] \, dt \right|  \nonumber \\
& \quad \leqslant \max \left\{ \frac{\kappa_1^2}{\delta}; \, \frac{\kappa_2^2}{\delta}; \, \frac{1}{2\delta} \right\} \int_r^T \Big( u_t^2(1,t) + u_t^2(1,t-\tau) + u_{xt}^2(1,t) \Big) \, dt \nonumber \\
& \qquad  + \delta \max\left\{ 2C_1^2; \, \frac{1}{q_0} \right\}  \left( 2 C_1^2 + \frac{1}{q_0}  \right) \int_r^T \Big( u^2(1,t) + u^2_x(1,t) \Big) \, dt \nonumber \\
& \quad \leqslant \delta \max\left\{ 2C_1^2; \, \frac{1}{q_0} \right\}  \left( 2 C_1^2 + \frac{1}{q_0}  \right) \int_r^T \Big( u^2(1,t) + u^2_x(1,t) \Big) \, dt \nonumber \\
& \qquad + \frac{\max \left\{ \frac{\kappa_1^2}{\delta}; \, \frac{\kappa_2^2}{\delta}; \, \frac{1}{2\delta} \right\}}{C_{\kappa_1, \kappa_2}^\gamma} \Big( E(r) - E(T) \Big).
\end{align}

We now estimate $\displaystyle \int_r^T \int_0^1 u_t y_t \, dx  \, dt$. For this purpose, we consider the following problem:
\begin{gather} \label{eq:144}
\left\{
\begin{array}{l}
\big( \sigma(x) (y_t)_{xx} \big)_{xx} - \big(q(x) (y_t)_x \big)_x = 0, \\
q(1) (y_t)_x(1,t) - \big(\sigma(x) (y_t)_{xx} \big)_x(1,t) = u_t(1,t), \\
\sigma(1) (y_t)_{xx}(1,t) = u_{xt}(1,t).
\end{array}
\right.
\end{gather}
By Proposition \ref{p5}, $y_t(t, \cdot)$ solves the system \eqref{eq:144}. Consequently, the following estimate holds:
\begin{gather} \label{eq:145}
	\lVert y_t \rVert^2_{L^2(0,1)} \leqslant \frac{2}{q_0} \max\left\{ 2C_1^2; \, \frac{1}{q_0} \right\} \Big( u^2_t(1,t) + u^2_{xt}(1,t) \Big).
\end{gather}
Thus for any $\tilde{\delta} > 0$, we have:
\begin{align} \label{eq:146}
 \int_r^T \int_0^1 \left| u_t y_t \right| \, dx \, dt & \leqslant \frac{\tilde{\delta}}{2} \int_r^T \int_0^1 u_t^2 \, dx \, dt + \frac{1}{2\tilde{\delta}} \int_r^T \int_0^1 y_t^2 \, dx \, dt \nonumber \\
 \int_r^T \int_0^1 \left| u_t y_t \right| \, dx \, dt  & \leqslant  \tilde{\delta} \int_r^T E(t) \, dt + \frac{ \max\left\{ 2C_1^2; \, \frac{1}{q_0} \right\} }{\tilde{\delta} q_0  C_{\kappa_1, \kappa_2}^\gamma} \Big( E(r) - E(T) \Big).
\end{align}
Furthermore, taking into account inequalities \eqref{eq:140}, \eqref{eq:143} and \eqref{eq:146}, the relation \eqref{eq:139} becomes:
\begin{align} \label{eq:147}
\begin{split}
& \int_r^T \left( u^2(1,t) + u_x^2(1,t) \right) \, dt 
\leqslant \tilde{\delta} \int_r^T E(t) \, dt + C_3 \Big( E(r) + E(T) \Big) \\
& \quad + \left( \dfrac{ \max\left\{ 2C_1^2; \, \frac{1}{q_0} \right\} }{\tilde{\delta} q_0  C_{\kappa_1, \kappa_2}^\gamma} + \dfrac{\max \left\{ \frac{\kappa_1^2}{\delta}; \, \frac{\kappa_2^2}{\delta}; \, \frac{1}{2\delta} \right\}}{C_{\kappa_1, \kappa_2}^\gamma} \right) \Big( E(r) - E(T) \Big) \\
& \quad + \delta \max\left\{ 2C_1^2; \, \frac{1}{q_0} \right\}  \left( 2 C_1^2 + \frac{1}{q_0} \right)  \int_r^T \Big( u^2(1,t) + u^2_x(1,t) \Big) dt.
\end{split}
\end{align}
Finally, by choosing $\delta$ satisfying \eqref{eq:125a}, we get \eqref{eq:124}.
\end{proof}

We are now ready to state the second main result of this paper.
\begin{theorem} \label{t2}
Assume that the function $\sigma$ is (WD) or (SD). Let $u$ be a solution of  system \eqref{eq:p1}. Under assumptions \eqref{eq:0}, \eqref{eq:19} and \eqref{eq:72}, the energy $E(t)$ of system \eqref{eq:p1} decays exponentially to zero, i.e.:
\begin{gather} \label{eq:148}
E(t) \leqslant E(0) \, e^{1 - \dfrac{t}{M}}, \quad \forall t \in [M; \, +\infty),
\end{gather}
where the constant $M > 0$ is given in \eqref{eq:152a} and is independent of $(u_0, \, u_1)$.
\end{theorem}

\begin{proof}
Let $T>0$. By Propositions \ref{p2} and \ref{p4}, and Lemma \ref{l2}, we obtain \eqref{eq:102}. By using Lemma \ref{l3}, it follows that for all $r \in (0, T)$:
\begin{align*}
& \varepsilon \min\Big\{ 2 - \iota_{\sigma, q}; \, 4e^{-2\tau} \Big\} \int_r^T E(t) \, dt \leqslant L(r) - L(T) \\
& \quad + 2 \varepsilon C_0 \left[ \tilde{\delta} \int_r^T E(t) dt + C_2^\delta \Big( E(r) - E(T) \Big) + C_3 \Big( E(r) + E(T) \Big) \right]. 
\end{align*}
So, we have:
\begin{align} \label{eq:150}
& \varepsilon \left(  \min\Big\{ 2 - \iota_{\sigma, q}; \, 4e^{-2\tau} \Big\} - 2  \tilde{\delta} C_0 \right) \int_r^T E(t) \, dt \nonumber \\
& \quad \leqslant L(r) - L(T) + 2 \varepsilon C_0 \left[ C_2^\delta \Big( E(r) - E(T) \Big) + C_3 \Big( E(r) + E(T) \Big) \right].
\end{align}
Choosing $\tilde{\delta}$ such that:
\begin{align} \label{eq:151}
	0 < \tilde{\delta} < \frac{1}{2 C_0} \min\Big\{ 2 - \iota_{\sigma, q}; \, 4e^{-2\tau}\Big\},
\end{align}
the following estimates holds:
\begin{align*}
\int_r^T E(t) \, dt 
& \leqslant \varepsilon^{-1} \left( \min\Big\{ 2 - \iota_{\sigma, q}; \, 4e^{-2\tau} \Big\} - 2  \tilde{\delta} C_0 \right)^{-1} \Bigg[ L(r) - L(T) \\
& \qquad + 2 \varepsilon C_0 \Bigg( C_2^\delta \Big( E(r) - E(T) \Big) + C_3 \Big( E(r) + E(T) \Big) \Bigg) \Bigg].
\end{align*}
Next, using the inequality \eqref{eq:73} of Proposition \ref{p3}, we obtain:
\begin{align} \label{eq:152}
\Theta_1 E(r) - \Theta_2 E(T) \leqslant L(r) - L(T) \leqslant \Theta_2 E(r) - \Theta_1 E(T).
\end{align}
Hence, we get:
\begin{align}
& \int_r^T E(t) \, dt \leqslant M \, E(r),
\end{align}
where the constant $M$ is defined by:
\begin{gather} 
	M := \varepsilon^{-1} \left( \min\Big\{ 2 - \iota_{\sigma, q}; \, 4e^{-2\tau} \Big\} - 2  \tilde{\delta} C_0 \right)^{-1} \Big(
\Theta_2 + 2 \varepsilon C_0 C_2^\delta + 4\varepsilon C_0 C_3
 \Big). \label{eq:152a}
\end{gather}
Finally, by virtue of Theorem \ref{kormonik}, we obtain \eqref{eq:148}.
\end{proof}

 With the preceding theorem, we provide in the following remarks some additional insights into how the strength of the degeneracy and the time delay influence the final exponential decay rate.
 
\begin{remark} 
From Theorem \ref{t2}, system \eqref{eq:p1} is exponentially stable regardless of the type of degeneracy (WD or SD). This stability is uniform, as the decay rate is independent of the initial data $u_0$ and $u_1$. However, there are significant quantitative differences. The constants $C_1, \, C_2^\delta$, and $C_3$, defined respectively in \eqref{eq:108}, \eqref{eq:125} and \eqref{eq:126}, depend explicitly on the strength of the degeneracy through the parameter $\iota_\sigma$. Under the condition :
\begin{align}
	\sigma(1)(2 - \iota_\sigma) < q_0, 
\end{align}
these constants are larger in the (SD) case than in the (WD) case. Consequently, according to the expression of $M$, defined in \eqref{eq:152a}, the beam's vibrations dampen faster in the weakly degenerate (WD) case than in the strongly degenerate (SD) case.
Hence, a stronger degeneracy slows down the transmission of energy toward the dissipative boundary.
\end{remark}

\begin{remark}
While system \eqref{eq:p1} remains exponentially stable for any $\tau > 0$, the decay rate is intrinsically linked to the delay, as evidenced by the expression of the constant $M$. Specifically:
\begin{itemize}
\item when $\tau \to +\infty$, the constant $M \to \infty$, and the energy estimate \eqref{eq:148} becomes
\begin{align}
	E(t) \leqslant e E(0).
\end{align}
for large values of time $t$. Therefore, system \eqref{eq:p1} remains stable but its practical performance is significantly degraded (loss of exponential decay property);
\item when $\tau \to 0$, the constant $\bar{\delta}$ satisfies the relation $ 0 < \bar{\delta} < \dfrac{2 - \iota_{\sigma, q}}{2 C_0}$. Then, relation \eqref{eq:152a} becomes:
\begin{align}
	M = \varepsilon^{-1} \left( 2 - \iota_{\sigma, q} - 2 \bar{\delta} C_0 \right)^{-1} \left( \Theta_2 + 2 \varepsilon C_0 C_2^\delta + 4\varepsilon C_0 C_3 \right),
\end{align}
leading to a much larger and more effective decay rate. Thus, system \eqref{eq:p1} is far more efficient when the time delay $\tau$ is sufficiently small.
\end{itemize}
\end{remark}

\section{Conclusion and discussions}\label{sec5}

We have rigorously analyzed the stability of a complex Euler-Bernoulli beam model, which simultaneously incorporates a degenerate flexural rigidity, axial force and a time-delay boundary input. 
We established two main results under the assumption that the non-delayed feedback gain $\kappa_1$ dominates the delayed component, that is $\kappa_1 > |\kappa_2|$.
First, we proved the well-posedness of system \eqref{eq:p1} by recasting it as an abstract evolution problem. Using L\"umer-Phillips theorem, we showed that the associated linear operator generates a $\mathcal{C}_0$-semigroup of contractions on a dedicated weighted Hilbert state space.
Next, we established the uniform exponential stability of the closed-loop system. 
To account for the simultaneous presence of degenerate rigidity, non-uniform axial force, and boundary delay, we constructed a suitable Lyapunov functional which allows us to show that the energy of the system exponentially decays to zero, providing  a precise decay rate estimate independent of the initial data.

Building on these results, several interesting perspectives emerge for future research.
While the present study establishes exponential stability of system \eqref{eq:p1} under the condition $\kappa_1 > |\kappa_2|$, the case where the delay gain compensates or exceeds the direct damping ($\kappa_1 \leqslant |\kappa_2|$) remains an open and challenging problem. These specific situations, which may lead to instabilities or lack of exponential decay, will be the subject of future investigations using a frequency domain approach as in \cite{serge2006, serge2016}. In addition, we will conduct a numerical analysis of system \eqref{eq:p1} to precisely quantify the sensitivity of the decay rate in the weak and strong degeneracy cases, with respect to the parameters $\tau, \, \kappa_1$ and $\kappa_2$. 
Moreover, it would be of great interest to extend this analysis to system \eqref{eq:p1} in the non-degenerate case or, as in the recent works  \cite{akil2025a, akil2025b}, by incorporating a distributed singular damping term  $c(\cdot) u_t$ in the governing equation, aiming to understand how the singularity of the damping at the endpoints interacts with the degeneracy of the flexural rigidity and time-delay effects.

\nocite{*}
\bibliographystyle{unsrtsiam}
\bibliography{biblio}

\end{document}